\documentclass[12pt, leqno]{article}

\usepackage{amsmath,amssymb,amsthm}
\usepackage{latexsym}
\usepackage[small,nohug,heads=littlevee]{diagrams}
\diagramstyle[labelstyle=\scriptstyle]
\usepackage[all]{xy}

\usepackage[dvips]{graphicx}
\newarrow{Into} C--->

\newcommand{\lowerarrow}[1]{%
\setlength{\unitlength}{0.03\DiagramCellWidth}
\begin{picture}(0,0)(0,0)
\qbezier(-28,-4)(0,-18)(28,-4)
\put(0,-14){\makebox(0,0)[t]{$\scriptstyle {#1}$}}
\put(28.6,-3.7){\vector(2,1){0}}
\end{picture}
}

 \newtheorem{thm}{Theorem}
 
 \newtheorem{lem}{Lemma}
  \newtheorem{sublem}{Sublemma}
 \newtheorem{conj}{Conjecture}
 
 \newtheorem{defi}{Definition}

 \theoremstyle{remark}
 \newtheorem{rem}{Remark}[section]

\newcommand{\e}{\leqno}

\newcommand {\bi}[1]{\textbf{\textit {#1}}}

\usepackage[dvips]{graphicx}
\usepackage{graphics}
\usepackage{amsmath}
\usepackage{amscd}
\usepackage{amssymb}

\title{Finitely presented groups and the Whitehead nightmare}

\author{Daniele Ettore Otera\footnote{Mathematics and Informatics Institute
 of Vilnius University, Akademijos st. 4, LT-08663, Vilnius, Lithuania.
e-mail: daniele.otera@mii.vu.lt \ - \  daniele.otera@gmail.com}
 \textsc{ and } Valentin Po\'enaru\footnote{Professor Emeritus at
Universit\'e Paris-Sud 11, D\'epartement de Math\'ematiques,
B\^atiment 425, 91405 Orsay, France. e-mail: valpoe@hotmail.com }}

\date{}

\begin{document}

\maketitle

\begin{abstract}

We define a ``nice representation'' of a finitely presented group $\Gamma$ as being a non-degenerate essentially surjective simplicial map $f$ from a ``nice'' space $X$ into a 3-complex  associated to a presentation of $\Gamma$, with a strong control over the singularities of $f$, and such that  $X$ is \textsc{wgsc} (\textit{weakly geometrically simply connected}), meaning that it admits a filtration by simply connected and compact subcomplexes.
In this paper we study  such representations for a very large class of groups, namely \textsc{qsf}  (\textit{quasi-simply filtered})  groups,  where \textsc{qsf} is a topological tameness condition of groups that is similar, but weaker, than \textsc{wgsc}. In particular, we prove that any \textsc{qsf} group admits a \textsc{wgsc}  representation which is locally finite, equivariant and whose double point set is closed.

\vspace{0.1cm} \noindent {\bf Keywords:} finitely presented groups, weak geometric simple connectivity, quasi-simple filtration, universal covering spaces, singularities.

\vspace{0.1cm}  \noindent {\bf MSC Subject:} 57M05; 57M10;  57N35.
\end{abstract}

\section{Introduction}

The present paper deals with finitely presented groups satisfying a rather mild tameness condition, the \textsc{qsf} \textit{property} introduced and studied by  Brick and  Mihalik in \cite{BM1}. Roughly speaking, a space
is \textsc{qsf} (\textit{quasi-simply filtered}) if any compact subspace of it can be
``approximated'' by an (abstract) simply connected compact space; in particular this means that a \textsc{qsf} space admits a ``quasi-filtration'' by compact and simply connected subspaces. We may also give for it the following equivalent definition (that is actually a theorem of the first author (D.O.) with L. Funar \cite{FO2}): the finitely presented group $\Gamma$ is \textsc{qsf} if and only if there exists a smooth compact manifold
$M$ such that $\pi_1 M= \Gamma$ and whose universal cover $\widetilde M$ is
\textit{weakly geometrically simply connected} (\textsc{wgsc}), meaning that $\widetilde M$ admits a filtration by compact and simply connected submanifolds. But one should keep in mind that this simpler condition is
\textbf{not} group-presentation invariant, while the original
definition of Brick and Mihalik, which we will define properly  in the next
section, is. This is actually one of the important virtues of the
concept \textsc{qsf}.

\vspace{0.2cm}

Although usually a (topological) presentation of a finitely presented group
$\Gamma$ is just  any finite simplicial complex $K$ such that
$\pi_1 K=\Gamma$,  in this present paper, we will be very choosy for them (see Section 2). Actually, all  along the paper, a \textit{presentation} of $\Gamma$   will  \textbf{always} be a particular 3-dimensional complex obtained by a suitable thickening of a 2-complex (or, in other words, a compact 3-manifold with \textbf{singularities}, locally as (the wedge of three lines)$\times \mathbb R^2$; this is enough for
catching all finitely presented  groups. Let us denote by $M(\Gamma)$ such
a 3-complex with $\pi_1 M(\Gamma) =
\Gamma$, and by $\widetilde {M(\Gamma)}$ its universal cover (we will not recall anymore they are 3-dimensional spaces).

Our basic tool for dealing with finitely presented groups  will be the notion of 
\textsc{wgsc-representation}, which we will define formally in
the next section. It will suffice to say, for right now, that
contrary to the more usual group representations which, for a group 
$\Gamma$, take the general form ``$\Gamma \to$ something'', our
\textsc{representations}, which we will always write in
capital letters, take the dual form ``some space $X$ with
special features $\overset{f}{ \longrightarrow }
\widetilde {M(\Gamma)}$'' (and note that the universal covering space
$\widetilde{M(\Gamma)}$ is the same thing as the group $\Gamma$,
up to quasi-isometry).

This triple, $X \overset{f}{ \longrightarrow } \widetilde
{M(\Gamma)}$, is endowed with the following properties: $X$ is a 
  simplicial complex which is
\textit{weakly geometrically simply connected} (\textsc{wgsc}), $f$ is a
\textbf{non-degenerate} simplicial map, meaning that
$f(d$-simplex$)=d$-simplex, and, furthermore, the map $f$ is 
\textit{zippable}, by which we intend that the ``smallest''
equivalence relation on $X$ which is compatible with $f$ and which
is also such that the quotient space immerses into $\widetilde
{M(\Gamma)}$, via the obviously induced map, is the trivial
equivalence relation induced by $f$ itself, namely: $x\sim y
\Longleftrightarrow fx=fy$. In other words, what zippability means
is that the ``cheapest'' way to kill all the \textit{singularities} of $f$ (that are the points  $x\in X$ where $f$ is not locally an embedding,
i.e. the non-immersive points of $X$), is to kill all the double
points of $f$; and this will actually happen via folding maps.

\vspace{0.2cm}
Here are some additional explanations concerning this 
definition, which was given now rather informally, and which will
be restated rigorously in the next section. First of all,  the notion of 
\textit{weak geometric simple connectivity} (\textsc{wgsc}) has been introduced  by L. Funar (see \cite{FG}), and also studied by him and the first author (D.O.), in the  context of geometric group theory (see \cite{FO2}). 

\begin{defi}\label{wgsc}
A locally finite simplicial complex  $X$  is  said to be  {\em {weakly geometrically simply
connected}} (\textsc{wgsc}), if it has an
exhaustion by finite (compact) and  simply connected subcomplexes
 $K_1 \subset K_2 \subset \ldots \subset  X=\cup_i K_i$.
\end{defi}

But when it comes to groups, \textsc{wgsc} is not the more appropriate condition, since it is not presentation independent, and the good one is the \textsc{qsf} property  introduced by  Brick and
Mihalik \cite{BM1}.

\begin{defi}\label{qsf}
 A locally compact simplicial complex $X$ is \textsc{qsf} (i.e. \textbf{quasi-simply filtered}) if  for any compact
subcomplex $k \subset X$ there is a simply-connected compact
(abstract) complex $K$ endowed with an inclusion $k \overset {j}
{\longrightarrow } K$ and with a simplicial map $K \overset {f}
{\longrightarrow } X$ satisfying the  {\em{``Dehn condition''}}:  $M_2(f) \cap
j(k) = \emptyset$ (where $M_2(f)\subset K$ denotes the set of
double points of $f$), and entering in the following  commutative diagram:
$$\begin{diagram}
k &&\rInto^j(4,2)&&K \\
    &\rdTo_i& &\ldTo_f\\& &X
\end{diagram} \e (1) $$
where  $i$ is the canonical injection.
\end{defi}

Then, being \textsc{qsf} means that the space $X$ can just be ``approximated'' by a filtration of simply connected and compact subspaces; but what one gains is very valuable, since, unlike \textsc{wgsc},
\textsc{qsf} turns out to be a group theoretical,
presentation-independent notion: if $K_1, K_2$ are two presentations
 (i.e. presentation complexes) for the same finitely presented 
  group $\Gamma$, then $\widetilde K_1 \in \mbox
{\textsc{qsf}} \iff \widetilde K_2 \in\mbox {\textsc{qsf}}$ \cite{BM1}, and in such a case 
  $\Gamma$ is said \textsc{qsf}.  Recently, in the group theoretical context,
L. Funar and the first author (D.O.) \cite{FO2}    have proved
that  if $\Gamma$ has a presentation
$K_1$ such that $\widetilde K_1 \in $ \textsc{qsf}, then it also
has a presentation $K_2$ such that $\widetilde K_2 \in $
\textsc{wgsc}.

\vspace{0.2cm}

Coming back to the non-degeneracy of $f$,  note that this means, among other things, that the dimension of the \textsc{representation} space $X$, source of $f$, is
restricted to dim $X\leq 3$; and the only serious cases are actually
dim $X=2$ and dim $X=3$, each interesting in its own right.

So, we will speak about $2^d$-  and
$3^d$-\textsc{representations}, and the capital letters should
remind the reader that we are not talking about the mundane group
representations, where the dimension of the representation means
quite a different thing. Retain also that our
\textsc{wgsc-representations} $X \overset{f}{\longrightarrow }
\widetilde {M(\Gamma)}$ are sort of {\em{resolutions}} of $\widetilde {M(\Gamma)}$ into a \textsc{wgsc} space $X$.

\vspace{0.2cm}

With all these things, here is our main result, whose statement
 will become more precise (see Theorem  \ref{thm2}) in the next sections, after the \textsc{representations} are more formally defined.

\begin{thm}\label{thm0v1}

For any finitely presented \textsc{qsf} group $\Gamma$, there exists a
$2^d$-\textsc{wgsc-representation} $X^2 \overset{f}{ \longrightarrow }
\widetilde {M(\Gamma)}$, where the simplicial
complex $X^2$ is locally-finite and  \textsc{wgsc}, such that:
\begin{enumerate}
\item[(i)] both $f(X^2) \subset \widetilde {M(\Gamma)}$ and the double point 
set $M_2(f)\subset X^2$ are \textbf{closed} subsets;

\item[(ii)] moreover, one can get an $X^2$ with a
free $\Gamma$-action $\Gamma \times X \to X$ such that $f$ is
\textbf{equivariant}, i.e. $f(\gamma x)=\gamma x$ for all $\gamma
\in \Gamma$, $x\in X$.
\end{enumerate}

\end{thm}

\begin{defi}\label{easy representations}
A \textsc{representation} satisfying condition (i) above will be called \textbf{easy}. 
\end{defi}

Thus, we may also rephrase Theorem \ref{thm0v1} as follows:  
finitely presented \textsc{qsf} groups admit easy \textsc{wgsc-representations}.

Since \textsc{wgsc} is a weak and simplified version of the more known \textsc{gsc} (\textit{geometric simple connectivity})  concept, which stems from  differential
topology, and which concerns handle decompositions without handles of index 1 (for more on this important notion see e.g. \cite{FG, FO2, Ot_surv, Ot-Po-Ta, Po2-duke, Po_QSF2}),  we propose, at least informally, the following definition and the associated conjecture:

\begin{defi}\label{easy groups}
A finitely presented group $\Gamma$ is \textbf{easy} (or {\bf {easily-representable}}) 
if  it admits a 2-dimensional \textsc{gsc-representation} $X^2 \overset{f}{ \longrightarrow }
\widetilde {M (\Gamma)}$ (namely with a \textsc{gsc} $X^2$) which is easy, in the sense just defined (i.e. with closed $f(X) \subset \widetilde {M(\Gamma)}$ and
$M_2(f)\subset X^2$).
\end{defi}

\begin{conj}
All finitely presented \textsc{qsf} groups are easy (i.e. easily-representable).
\end{conj}

\begin{rem}

\begin{enumerate}

\item  The converse implication is already a theorem proved by
the authors in  \cite{Ot-Po1}.

\item The second author
(V.P.) has developed a program \cite{Po_QSF1_Geom-Ded, Po_QSF2, Po_QSF3}  aiming to prove  that \textbf{all} finitely presented groups  are \textsc{qsf}.

\item From papers like \cite{Po2-duke, Po3-JDG, Po5-Top2}, it 
can be extracted a proof of the following general form:  if  a  finitely presented group  satisfies a nice geometric condition (like e.g. Gromov-hyperbolicity, almost-convexity, automaticity, combability etc), then it is easy.

\end{enumerate}
\end{rem}

Coming back now to the \textsc{representations}, with which this
paper deals, they were
already present in the paper \cite{Po4-Top1}, where homotopy 3-spheres $\Sigma^3$  were \textsc{represented}.
Then, in \cite{Po2-duke, Po3-JDG, Po5-Top2, Po-contMath}, 
 universal covering spaces of  compact 3-manifolds have been 
\textsc{represented}, while in \cite{PT-Wh},
\textsc{representations} of the classical \textit{Whitehead
manifold} $Wh^3$ \cite{Wh} were investigated and what was found
there, was that for the simplest and most natural
$2^d$-\textsc{representations} of it, $X^2 \overset {f}
{\longrightarrow } Wh^3$, the  $M_2 (f) \subset X^2$ is
\textbf{not} a closed subset. (Hence, the $M_2 (f)$ being closed can be viewed as on obstruction  for a complex to admit a cocompact free action of an infinite group).

Note also that the notion
of \textsc{representation} of a group $\Gamma$, $X
\overset {f} {\longrightarrow } \widetilde {M(\Gamma)}$, 
has a-priori nothing group-theoretical about it, except that it allows the possibility
of a free action $\Gamma \times X \longrightarrow  X$, with an
\textit{equivariant} $f$, i.e. $f(gx) = gf(x)$; the point (ii) in Theorem
\ref{thm0v1} brings this option to life.

In the next section we will state more formally, and with more
details, what the paper actually proves. Then Theorem \ref{thm0v1}
will appear as a piece of some bigger, more comprehensive
statement. This will deal with $3$-dimensional \textsc{representations} too,
and then the ``\textit{Whitehead nightmare}'' appearing in the
title of this paper will be explained too.

\vspace{0.3cm}

\textbf{Acknowledgments:} We wish to thank Louis Funar, David
Gabai, Fr\'ed\'eric Haglund and Barry Mazur for very useful
conversations and remarks. In particular Louis Funar helped us to fix up a first version of our paper. We are also very grateful to the referee for his corrections and for his 
appropriate and constructive suggestions  to improve the readability of the paper.

 The first author (D.O.) was funded
by a grant from the Research Council of Lithuania (Researcher teams' project No. MIP-046/2014/LSS-580000-446). He also thanks the
Mathematics Department of the University of Paris-Sud 11 for the
hospitality during  research visits.

\section{Definitions and statements of the results}

We will give now, with full details, the definition of the
\textsc{representations} for finitely presented groups $\Gamma$,
which were only very informally presented in the last section.

To begin with, like in \cite{Po_QSF1_Geom-Ded},
we consider (topological) presentations for $\Gamma$ which are singular compact 3-manifolds
with non-empty boundary, denoted by  $M(\Gamma)$. The structure of such an
$M(\Gamma)$ is very simple (see e.g.  \cite{Ot-Po1}).
Start with a compact 3-dimensional
handlebody of some appropriate genus $g$, call it $H$; this
embodies the generators of the group $\Gamma$. Then 2-handles are
attached to $H$, embodying the relations of $\Gamma$. Explicitly, the attaching
zones are given by an \textbf{immersion}
$$ \sum_{j=1} ^k (S_j^1 \times [0,1])
\overset {\phi} {\longrightarrow} \partial H , \e{(2)}$$
which injects on each individual $S^1_j \times I$, the
double points coming from (singular) little squares $S \in
\partial H$, where $\phi (S^1_l \times I)$ and $\phi (S^1_m \times
I)$, for $m\neq l$, go through each other. These \emph{immortal
singularities} $S$ are the points where $M(\Gamma)$ fails to be a
3-manifold.

Now we are ready to give the precise and formal definition of
\textsc{wgsc-representations} for finitely presented groups, leaving
more details and comments just after the definition.

\begin{defi}\label{repr}
A \textsc{wgsc-representation}  of a finitely presented group $\Gamma$ is a simplicial 
map $$ X \overset {f} {\longrightarrow} \widetilde {M(\Gamma)}, \e{(3)}$$
where $\widetilde {M(\Gamma)}$ is the universal cover of $M(\Gamma)$, which satisfies the following list of conditions:

\begin{itemize}
\item[(3$-$1)] the space $X$ is a countable simplicial complex
which is \textbf{not} necessarily assumed to be locally-finite; but it is assumed to be \textbf{weakly geometrically simply connected} ({\sc wgsc});

\item[(3$-$2)] the simplicial map $f$ is \textbf{non-degenerate}, which also means that
$\dim X \leq 3$. Hence, once the meaningless case $\dim X=1$ is
discarded, we are left with the two meaningful cases $\dim X=2$
and $\dim X=3$, namely with 2- and 3-dimensional 
\textsc{representations};

\item [(3$-$3)] the equality $\Psi (f) = \Phi (f)$ holds (see the
explanation here below), and in this case we say  that $f$ is
\bi{zippable};

\item [(3$-$4)] the map $f$ is ``essentially surjective'',
which means the following: if $\dim X=3$, then
$\overline{\mbox{Im}\ f}= \widetilde {M(\Gamma)}$, and if $\dim
X=2$, then $\widetilde {M(\Gamma)} = \overline{\mbox{Im}\ f} + \{
\mbox {cells of dimension 2 and 3}  \}$.
\end{itemize}
\end{defi}

Here, some remarks and details are needed. First of all, concerning the point (3-3) above, consider a
non-degenerate simplicial map $g: A \to B$, like, for instance,
our map $f$ from (3); for any such a map we define the set of
\emph{mortal singularities}, $Sing (g) \subset A$, as being the
set of those points $x \in A$, at which $g$ fails to be immersive.
There  are two interesting equivalence relations on $A$, in this
context. To begin with, we have the trivial one  $ \Phi (g)
\subset A \times  A$,  where  $(x,y) \in \Phi (g)
\Longleftrightarrow gx = gy$. Then (and see here \cite{Po1-duke,
Po_QSF1_Geom-Ded} for more details) there is the following more
subtle equivalence relation $\Psi (g) \subset \Phi (g)$, which is
defined as follows (and it can be proved that this definition
makes sense, see \cite{Po1-duke}): $\Psi (g) \subset A \times A$
is the ``smallest'' equivalence relation compatible with $g$,
which kills all the mortal singularities, i.e. which is such that
in the following diagram the map $g_1$ is an immersion (i.e. $Sing
(g_1) = \emptyset$) $$\begin{diagram}
A &&\rTo^g(4,2)&&B \\
    &\rdTo_{\pi}& &\ruTo_{g_1}\\& & A / \Psi (g) .\end{diagram} $$
It can be shown that there is a \textbf{uniquely} well-defined equivalence
relation $\Psi (g)$ (constructed via a sequence of \textit{folding maps}) with
the properties listed above, and that it has the additional
property that the following induced map is \textbf{surjective} $$\pi_1 (A)
\overset {(g_1)_{\ast}} { \xrightarrow{\hspace*{1cm}} } \pi_1 (A / \Psi (g)).$$
Details concerning the equivalence relations $\Psi$ and $\Phi$ can
be found in \cite{Po1-duke, Po_QSF1_Geom-Ded}.

\begin{rem}
\begin{itemize}
\item  It should be stressed that, the general definition of
\textsc{representation} $X \overset{f}{\rightarrow} Y$ is such
that the object $Y$ which is \textsc{represented}, automatically
comes with $\pi_1 Y=0$.

\item For any finitely presented
group,  it can be shown that \textsc{representations}  
always exist \cite{Po_QSF1_Geom-Ded}; but usually, the simplest
\textsc{representations} which one stumbles upon \textbf{fail} to be locally finite. 
 On the other hand, in this paper, only \textsc{representations} of groups with a 
\textbf{locally-finite} $X$ will be considered.

\item 
Many other objects can be \textsc{represented},  provided they
are simply connected. The definition is always exactly the same,
but what is special when one represents groups, which comes
automatically with the canonical action $\Gamma \times \widetilde
{M(\Gamma)} \longrightarrow \widetilde {M(\Gamma)}$, is that there is
then the possibility that the \textsc{representation} $ X
\overset{f}{ \longrightarrow} \widetilde {M(\Gamma)}$ may be
{\textbf{\textit{equivariant}}}, meaning that there may be a second free
action $\Gamma \times X \longrightarrow X$, coming with $f(\gamma
x)= \gamma f(x)$ for all $\gamma \in \Gamma , \ x\in X$. 
\end{itemize}
\end{rem}

Without any additional assumption on the triple $ X
\overset{f}{ \longrightarrow} \widetilde {M(\Gamma)}$ from (3) above,
there is a \textbf{metric structure}, well-defined up to \textit{quasi-isometry},
which permeates this whole story. Chose any Riemannian metric on
$M(\Gamma)$, and what we mean by this is the following. On each
individual 3-dimensional handle $H^\lambda _i$ of $M(\Gamma)$, a
Riemannian metric is given and, whenever two handles are incident,
it is required that the induced metrics on the intersection should
coincide. Then, using the non trivial free group action $\Gamma
\times \widetilde {M(\Gamma)} \longrightarrow \widetilde {M(\Gamma)}$,
the arbitrarily chosen Riemannian metric on $ M(\Gamma)$ lifts to
an equivariant metric on $\widetilde {M(\Gamma)}$. Finally, one
lifts this metric on $X$, via the non-degenerate map  $f: X \to \widetilde {M(\Gamma)}$. Thus, $X$ becomes a
metric space and, up to quasi-isometry, this metric on $X$ is
\textit{canonical}, i.e. independent of the original choice of
Riemannian metric on $M(\Gamma)$.

Let us fix now a compact \textit{fundamental domain} $\delta \subset
\widetilde {M(\Gamma)}$, such that $\widetilde {M(\Gamma)} =
\bigcup _{\gamma \in \Gamma}  \gamma \delta$. 
In a similar vein, we consider ``\textit{large fundamental domains}'' $\Delta \subset X$,
and a locally finite decomposition of $X$ into such domains, 
$X = \bigcup _{j\in J} \Delta _j$, 
where $J$ is some countable set of indices. Since there is no
group action on $X$ (in the general case, at least), what we will
ask now from the compact pieces $\Delta_j$ above, apart from the
obvious condition that their interiors should be disjoined, is the
existence of two positive constants $C_2 > C_1 >0$ such that we
should have
$$C_1 \leq \| \Delta _j\| \leq C_2, \ \forall \ j\in J . \e (3-5)$$
Here $\| \Delta_j\|$ is the diameter of $\Delta_j$. Our large
fundamental domains $\Delta_j$ could be, for instance, maximal
dimensional cells of the cell-decomposition of the representation
space $X$ occurring in (3), satisfying the metric condition
(3$-$5), when $j\to \infty$. The next Theorem \ref{thm2}, stated
below, has two parts corresponding to the dimension of $X$, in a
3-dimensional \textsc{representation} this is $X=X^3$, while in a
2-dimensional \textsc{representation} it is $X=X^2$. In both cases
we have also \textit{immortal singularities}, 
\textit{Sing}$\big(\widetilde {M(\Gamma)}\big)\subset \widetilde {M(\Gamma)}$, 
and \textit{mortal singularities}, \textit{Sing}$(f)\subset X$.

At least in the 2$^d$ case, we will want to be a bit more specific
about the singularity issues, and so, when it comes to the 2-dimensional
part of the Theorem  \ref{thm2} stated below,
the following condition will be imposed too
$$ \mbox{the set of mortal singularities \textit{Sing}}(f)\subset X^2
\mbox { is \textbf{discrete} and, }\e (3-6)$$ at each $x\in$
{\textit{Sing}}$(f)$, there is the following local model. There is
an open neighborhood $P=P_1 \cup P_2$ of $x$ in $X^2$ and an
embedding $\mathbb R^3 \longrightarrow \widetilde {M(\Gamma)}$ \big(which,
a priori, might happily go through \textit{Sing} $ \widetilde
{M(\Gamma)}$\big),  through which $P \overset{f}{\longrightarrow}
\widetilde {M(\Gamma)}$ factorizes. At the source $X^2$, the $P_1,
P_2$ are two planes $\mathbb R^2$ glued along a half-line $[0,
\infty)$ with $x=0$, $x$ being here our mortal singularity.

In the diagram below
$$\begin{diagram}
P & &\rTo^j(4,2)& & \mathbb R^3 \\
&\rdTo_f& &\ldTo\\
& & \widetilde {M(\Gamma)}
\end{diagram} $$

\noindent each $j|_{P_1}, j|_{P_2}$ injects, the two being transverse.
So, there is a double line in $M_2(f)$ starting at the mortal
singularity $x$. This is a local model already used by the second
author (V.P.) in   \cite{Po4-Top1}, where,  according to a suggestion of Barry Mazur, these singularities were called ``\textit{undrawable}''.

For our \textsc{representation} (3), we will also assume that
$$ f\big(\mbox{\textit{Sing}}(f)) \cap \mbox{\textit{Sing}}
(\widetilde M(\Gamma)\big)= \emptyset . \e (3-7)$$
But, at the later stages in the zipping of $f$, this condition may
be violated. Then, besides the \textit{Sing}$(f) \subset X^2$,
there is also another set of immortal singularities, \textit{Sing}$(X^2)
\subset X^2 - \mbox{\textit{Sing}}(f)$, which is also \textbf{discrete}.
This comes with the inclusion $f\big(\mbox{\textit{Sing}}(X^2)\big) \subset
\mbox{\textit{Sing}}\big(\widetilde {M(\Gamma)}\big )$. At the points $x\in
\mbox{\textit{Sing}}(X^2)$, there are no local factorizations
$$\begin{diagram}
& & V \subset \mathbb R^3 \\
& \ruTo &  & \rdTo \\
X^2 \supset U & & \rTo^f(4,2)& & \widetilde {M(\Gamma)},
\end{diagram} $$
and it is their absence which makes the $x\in
\mbox{\textit{Sing}}(X^2)$ be an immortal singularity, never to be
killed by the zipping. But, in purely topological terms, and
forgetting about $f$, at one immortal singularity $x\in X^2$, the
$X^2$ looks exactly alike as at a mortal singularity. This ends
our digression on $\mbox{\textit{Sing}}(f)$.

\vspace{0.2cm}

We are now ready to state with all details the main result  of the
present paper.

\begin{thm}[\bf{Main Theorem}]\label{thm2}\
\begin{enumerate}

\item (3$^d$-part)  For any finitely presented \textsc{qsf} group $\Gamma$, there exists a
locally finite 3$^d$ \textsc{wgsc-representation}, $X^3 \overset
{f} {\longrightarrow} \widetilde {M(\Gamma)}$, such that the following conditions are satisfied for any $\gamma \in
\Gamma$:
$$ \mbox{ there is a free action } \Gamma \times X^3 \longrightarrow X^3 \mbox {, and }
 f \mbox { is \textbf{equivariant}};  \e(4)$$
$$ \mbox{ there is a constant } C = C(\delta) > 0  \mbox { s.t. }
\forall \gamma \in \Gamma \mbox { one has: } \# \{\Delta _i
\mbox { s.t. } f(\Delta_i) \cap \gamma \delta  \} < C. \e(5)$$

In particular, any given domain $\gamma \delta \subset \widetilde {M(\Gamma)} =
\bigcup _{\gamma \in \Gamma}  \gamma \delta$ 
downstairs, can only be hit \textbf{finitely many} times by the image of a large domain
$\Delta \subset X^3 = \bigcup _{j\in J} \Delta _j$ from upstairs.

\item (2$^d$-part) For any finitely presented \textsc{qsf} group $\Gamma$, there exists a
locally finite 2$^d$ \textsc{wgsc-representation} $X^2 \overset
{f} {\longrightarrow} \widetilde {M(\Gamma)}$ which is both
\textbf{equivariant}, like in (4), and which also satisfies the following condition:
$$ \mbox{ both \textit{Im}}(f)=f(X^2) \subset \widetilde {M(\Gamma)}
 \mbox { and also } M_2(f)\subset X^2 \mbox { are \textbf{closed} subsets}.
 \e(6)$$
\end{enumerate}
\end{thm}

\begin{rem}
For a generic 3$^d$-\textsc{representation}  $X^3 \overset {f}
{\longrightarrow} \widetilde {M(\Gamma)}$, one normally finds the
following situation, at the opposite pole with respect to our (5)
above, and which, in papers like \cite{Po-contMath}, the second
author (V.P.) has called the \textit{Whitehead nightmare}
$$ \# \big \{ \Delta _i , \mbox { s.t. } f(\Delta _i) \cap
\gamma \delta \neq \emptyset \big\} = \infty, \ \forall \gamma \in
\Gamma . \e (5^{\ast})$$
Our present Whitehead nightmare under discussion, should remind
the reader of the basic structure of the classical Whitehead
manifold $Wh^3$ \cite{Wh} (whence the name of our nightmare), of
the \textit{Casson Handle} \cite{Gu-Ma}, or of the \textit{gropes} of M.
Freedman and F. Quinn \cite{Fr-Qu}.
\end{rem}

So, the first part of our Theorem  means that \textsc{qsf}
finitely presented groups can \textbf{avoid} the
Whitehead nightmare, and this is what the title of the present
paper refers to.

The 2$^d$ counterpart of the Whitehead nightmare (5$^\ast$) is the
following condition
$$ M_2(f) \subset X^2  \mbox { is \textbf{not} closed}. \e (6^\ast)$$
This \textbf{is} the generic situation for
2$^d$-\textsc{representations} and one has to start by living with
it and look at the accumulation pattern of $M_2(f)$ inside $X^2$,
all this being studied in \cite{Po_QSF1_Geom-Ded, Po_QSF2, Po_QSF3}.

\begin{rem}
\begin{enumerate}

\item The representation spaces $X$ occurring in the two points 
above are, of course, distinct spaces, although not quite
totally unrelated, as we shall see.

\item Condition (5) of Theorem \ref{thm2} can also be replaced by the following
variant: there exist \textbf{equivariant triangulations} for
$\widetilde {M(\Gamma)}$ and for $X^3$, and also a constant $C'$
such that, for any simplex $\sigma \subset \widetilde {M(\Gamma)}$,
we should have
$$ \# \big \{\mbox {simplexes } S \subset X^3 \mbox {, s.t. }
f(S) \cap \sigma \neq \emptyset  \big\} < C'. \e(5-bis)$$

\end{enumerate}
\end{rem}

\section{Preliminaries lemmas}

We give now the beginning of the proof  of  Theorem 
\ref{thm2}.  Some technicalities will be postponed until the
next section. Since $\Gamma$ is \textsc{qsf}, this also means that
$\widetilde {M(\Gamma)} \in $ \textsc{qsf} (because $M(\Gamma )$ is a compact 2-complex associated to a presentation
of $\Gamma$). Since our 3-dimensional complex $\widetilde {M(\Gamma)}$ has singularities, we prefer to replace it by a smooth, albeit higher dimensional, object.

Let $\mathcal R$ be a \textit{resolution} of the singularities of
$\widetilde {M(\Gamma)}$ (and see here \cite{Po4-Top1}, or  better,
our recent joint work \cite{Ot-Po1}, where all this issue is
explained in a context which is very much akin to the present
one). Given a choice of $\mathcal R$, we get a smooth 4-manifold
$$\Theta ^4 (M(\Gamma), \mathcal R),\e(7)$$ and, as soon as one takes the
product with $B^m$, for $m\geq 1$, and one goes to $\Theta ^4
(M(\Gamma), \mathcal R) \times B^m$, then the $\mathcal
R$-dependence is washed away, and everything becomes then
\textbf{canonical}. In particular, there is now a free action of $\Gamma$ on
 $\Theta ^4 (\widetilde {M(\Gamma)}, \mathcal R) \times B^m$, for
 $m\geq 1$, and one has that
 $$  \big ( \Theta ^4 (\widetilde {M(\Gamma)},
 \mathcal R) \times B^n \big ) / \Gamma =  \Theta ^4 (M(\Gamma), \mathcal R)
\times B^n.  \e(8)$$

\noindent We take now $n= m +4 \geq 5$, and then we get the manifold 
$$  M^n\overset{\textsc{def}}
=   \Theta ^4 (\widetilde {M(\Gamma)},
 \mathcal R) \times B^{n-4} \equiv \mbox{The universal cover of }\big ( \Theta ^4 (M(\Gamma), \mathcal R)
 \times B^{n-4} \big).\e{(9)}$$

\noindent This $M^n$ is a smooth non-compact manifold, of very large boundary. Also,
because $\Gamma \in$ \textsc{qsf}, we also have $M^n \in $ \textsc{qsf}.

\begin{lem}\label{lemA}
If $N >> n$, then the manifold $W^p  \overset{\textsc{def}}
=  M^n \times B^N$, for $p = n+N$, is
\textsc{wgsc}.
\end{lem}

\begin{proof}
The proof of this lemma is rather standard  (see e.g. \cite{FG, FO2}). Hence we omit it.
\end{proof}

\begin{lem}\label{ends}
It suffices to prove Theorem \ref{thm2} in the case of  one-ended groups.
\end{lem}
\begin{proof}
When $e(\Gamma)=0$, then $\Gamma$ is finite, $\widetilde{M(\Gamma)}$ is compact, and the canonical \textsc{representation} $id: \widetilde{M(\Gamma)}  \to \widetilde{M(\Gamma)}$ (i.e. with $X=\widetilde{M(\Gamma)}$ and $f=id$) satisfies our Main Theorem for $\Gamma$. 

When $e(\Gamma)=2$, then we have a very good explicit description of $\Gamma$ (as finite extension of $\mathbb Z$), with which the Main Theorem for $\Gamma$ is easily proved, directly.

Finally, when $e(\Gamma)=\infty$, we need to appeal to the celebrated theorem of J. Stallings (see \cite{St_libro, Po_ends}), which tells us that $\Gamma$ is gotten by amalgamation from one or two groups $G$ with $e(G)=1$ and a finite group $F$ (with $e(F)=0$). Now, $\widetilde {M(G)}$ contains (a lot of) copies of $\widetilde {M(F)}$.   For each $G$ with $e(G)=1$, assuming that Theorem \ref{thm2} holds for one-ended groups, we have a \textsc{wgsc-representation} $f: X \to \widetilde {M(G)}$ like in the Main Theorem and it may be assumed that, for each $\widetilde {M(F)} \subset \widetilde {M(G)}$, the map $f|_{f^{-1} \widetilde {M(F)}}$ is the identity map. 

Now, we got the $\widetilde{M(\Gamma)}$ by taking infinitely many copies of the $G$'s, each coming with its \textsc{wgsc-representation} $f_i : X_i \to \widetilde{M(G)_i}$ like in the Main Theorem, and then $\widetilde{M(\Gamma)}$ is an infinite tree-like union of these $\widetilde{M(G)}_i$'s, glued along the common $\widetilde{M(F)}$'s.

With these things the following map $$\underset {f^{-1} \widetilde {M(F)}} {\bigcup} X_i \overset {\cup f_i}{ \xrightarrow{\hspace*{1cm}}} \underset{ \widetilde{M(F)}}{\bigcup} \widetilde{M(G)_i} =\widetilde{M(\Gamma)}\e (10)$$ is a \textsc{wgsc-representation} of $\Gamma$ satisfying Theorem \ref{thm2}. 
\end{proof}

Now, Lemma \ref{lemA} tells us that there is an exhaustion
by compact, simply-connected, codimension zero submanifolds,
each embedded in the interior of the next
$$  K_1 \subset K_2 \subset \cdots \subset W^p = \cup _{1} ^{\infty } K_i .\e (11)$$
While, by Lemma \ref{ends}, we can suppose  $e(\Gamma)=1$. Hence, $W^p$ has one end too, and so the sets   $$(\partial K_i - \partial W) \  \mbox{ are disjoined, and, for each  $i$, both}   \e(12)$$ $$(\partial K_i - \partial W) \ \mbox{ and } \ (K_{i+1} - K_i)\  \mbox{ are connected}.\e(13)$$

\begin{lem}\label{lemC}
Given $\Gamma$, we can chose our presentation $M(\Gamma)$ so that,
for any desingularization $\mathcal R$, the smooth 4-manifold $Y^4 \overset{\textsc{def}}
= \Theta ^4 \big(M(\Gamma), \mathcal R \big)$ is 
parallelizable and there is
a smooth submersion
$$ Y^4 
\overset {\phi_1} {\longrightarrow} \mathbb R^4.\e (14)$$
\end{lem}

\begin{proof}

Along each singular square $S \subset Sing \big(M(\Gamma)\big)$, the
$M(\Gamma)$ has three smooth branches
$$U_1 \subset H (= \mbox{the 3$^d$ handlebody}), \ U_2 \subset
D_{j_1}^2 \times [0,1], \ U_3 \subset D_{j_2}^2 \times [0,1],$$
coming with $S = S_{j_1}^1 \times [0,1] \cap S_{j_2}^1 \times
[0,1] \subset \partial H$, where $\partial D_j = S_j ^1$  \big (see (2)\big ).

Each of the $U_1 \cup U_2$ and $U_1 \cup U_3$ is a smooth
3-manifold, and,  for each $x
\in S$, there is a canonical identification $T_x (U_1 \cup U_2 ) =
T_x (U_1 \cup U_3)$, defining the $T_x \big(M(\Gamma)\big)$ for $x \in S$.
 For the smooth points of $M(\Gamma)$ this tangent space is
obvious.

\begin{sublem}\label{claim}
For each $\Gamma$, we can chose the
$M(\Gamma)$ so that there is a smooth submersion  into the Euclidean 3-space
$$M(\Gamma) \overset {\psi_0} {\longrightarrow } \mathbb R^3. \e(15)$$
\end{sublem}

\begin{proof}
Start with an arbitrary chosen
presentation for our $\Gamma$
$$M(\Gamma)_0 = H \cup \sum_{j=1}^k D_j ^2 \times [0,1]$$
where each $D_j ^2 \times [0,1]$ is glued to $H$  via the $\phi |_{
S_j ^1 \times [0,1]}$ in (2), with, of course, $S_j ^1 =
\partial D_j ^2$. Next, take any embedding $H \subset \mathbb
R^3$, the standard one if one wants, but it does not matter. If
$\phi (S_j ^1 \times [0,1])\subset \mathbb R^3$ extends now to a
submersion, we are ok, in the sense that our $H \subset \mathbb
R^3$ extends to a submersion of $H \cup D_j ^2 \times [0,1]$. If
\textbf{not}, we can change the embedding $S_j ^1 \times [0,1]
\subset H$ by letting it spiral around $H$ so that now we get a
regular homotopy class
$\phi | S_j ^1 \times [0,1] \longrightarrow \mathbb R^3$
which does extend to an immersion $D_j ^2 \times [0,1]
\longrightarrow \mathbb R^3.$

This process can be performed in such a way that the homotopy
class of $\sum _i ^k S_j ^1 \longrightarrow H $ should stay
unchanged.
Of course, more singularities $S$ get created, the $S_j ^1 \times
[0,1]$'s are only immersed and not embedded, but all this is ok.
\end{proof}

Sublemma \ref{claim} provides us with a smooth field of frames
$$ F^3(x) \in \big\{ \mbox{Frames of } T_x \big(M(\Gamma)\big) \big \}
\simeq SO(3), \e (16)$$

\noindent for each $x \in M(\Gamma)$. We consider now the composite map
$$ M(\Gamma) \overset {\phi_0} {\longrightarrow} \mathbb R^3 = \mathbb R^3 \times \{0\}
\subset  \mathbb R^4 = \mathbb R^3 \times (-\infty < t < +\infty), \e
(17)$$

\noindent starting from which, any desingularization $\mathcal R$ of
$M(\Gamma)$
$$ \{ U_2, U_3 \} \overset {\mathcal R_S} {\longrightarrow}
\{s,n\}, \e (18)$$ produces a smooth immersion
$$ M(\Gamma) \overset {\Phi _0 } { \longrightarrow}
\mathbb R^4,\e(19)$$ simply by pushing the $s$-branch in (18)
towards $t = +1$ and the $n$-branch towards $t=-1$. With this (as
explained in \cite{Ot-Po1, Po4-Top1}), we have
$$Y^4 \overset{\textsc{def}}
=  \Theta ^4 \big( M(\Gamma), \mathcal R\big) =
\{\mbox{the }4^d \mbox{ smooth regular neighborhood of }
 M(\Gamma), \mbox{ induced by } \Phi_0\}.\e(20)$$

\noindent At each $x \in M(\Gamma)$, the $(\Phi _0)_{\ast} F^3(x)$ (with
$F^3$ like in (16)) is a 3-frame of the tangent space
$T_{\Phi_0(x)} \mathbb R^4$. By adding appropriately a fourth
orthogonal vector, we can complete this 3-frame into an oriented
4-frame. This is then a trivialization of the tangent space
$T(Y^4)  | _{M(\Gamma)}$,  which then easily induces a parallelization for $Y^4$.  
Hence, via the $h$-principle for
immersions and/or submersions (which in this particular case
boils down to the standard Smale-Hirsh theory), it follows that there exists our 
 smooth submersion $\phi_1$. This ends Lemma \ref{lemC}.
\end{proof}

When the $\phi_1$ of (14) is extended to a larger version of
$Y^4$,
$$Y^4_1 \overset{\textsc{def}} = Y^4 \cup \big (\partial Y^4 \times [0,1) \big ) \supset Y^4$$
we get a locally finite \textbf{affine structure} on the extension
$Y_1 ^4$ of $Y_4$, i.e. a Riemannian (not necessarily complete)
metric with sectional curvature $K =0$. There exists also a second structure
on $Y^4$, namely a \textit{foliated structure}, to be described
next. Both the affine and the foliated structures are compatibles
with the natural \textsc{Diff} structure of $Y^4$.

Let $L_3 =  \partial Y^4$ and let
 consider the following natural retraction $r$ coming from (20):
$$\begin{diagram}[w=3em]
L^3  \subset   Y^4  \lowerarrow{r|_{L^3}} \overset{r} {\longrightarrow} \ M(\Gamma).
\end{diagram}\e (21)$$

\begin{lem}\label{lemD}\
\begin{enumerate}

\item The map $r|_{L^3}$ is simplicially non-degenerate, and, outside
of some very simple fold-type singularities, it is an immersion
into  $M(\Gamma)$.

\item There is an isomorphism
$$ (Y^4, L^3) \simeq \big (  M(\Gamma) \underset {L^3 \times \{0\}}
{\cup} L^3 \times [0,1], L^3 \times \{1\}\big ), \e(22)$$ where
the map $r|_{L^3 \times \{0\}}$ is used for glueing together
$M(\Gamma)$ and $L^3 \times [0,1]$.
\end{enumerate}

\end{lem}

\noindent The proof is trivial, and  we left it to the reader. Lemma \ref{lemD} tells us that $Y^4$ admits a codimension-one \textbf{foliation}
$\mathcal F$, given by
$$ Y^4 = \bigcup _{t\in [0,1]} L^3_t,$$

\noindent where, for $t > 0$, we have $L_t ^3 = L^3$ and where $ L_0 ^3
\equiv M(\Gamma)$ is the unique \textbf{singular leaf}.

Returning now to the affine structure which $\phi_1$ (14) induces
on $Y_1 ^4 \supset Y^4$, we endow $\mathbb R^4$ with a very fine
affine triangulation, which we afterwards pull back on $Y^4_1$, so that
$L^3 = \partial Y^4$ becomes a polyhedral hypersurface. Next, with an appropriate
$N_1 \in \mathbb Z _+$, in the context of the Lemma \ref{lemA}, we
have that
$$M^n \times B^N= (\widetilde Y^4 \times B^{N_1}) = \widetilde {Y^4 \times B^{N_1}}$$
and $B^{N_1} = [0,1]^{N_1}$ has its own canonical affine structure, putting now
affine structures on $Y^4 \times B^{N_1}$ and on $M^n \times B^N$ (here $M^n$
is like in (9)).

Remember that Lemma \ref{lemA} tells us that there is a \textsc{wgsc} cell decomposition  of $M^n \times B^N$, call it $H(0)$. Without
any loss of generality, there is an affine triangulation $\Theta$ of $Y^4 \times B^{N_1}$
such that $H(0) \equiv \widetilde \Theta \overset{\textsc{def}} =  \{$the lift of $\Theta$ from $Y^4 \times B^{N_1}$
to $\widetilde {Y^4 \times B^{N_1}}\}$.

Our strategy will be now to work downstairs, at the level of $Y^4 \times B^{N_1}
 \overset {\pi} {\longrightarrow} Y^4$ and use only \textit{admissible subdivisions} for our
 cell-decompositions (by admissible we mean subdivisions which are baricentric or stellar or Siebenmann bisections \cite{Sie_bis}). When we will lift these things, afterwards, at the level
 $M^n \times B^{N} = \widetilde {Y^4 \times B^{N_1}}$, equivariance will be automatic,
 the admissible condition, which is local, is verified upstairs too, and there it will
 preserve the \textsc{wgsc} property which $H(0) = \widetilde \Theta$ initially had.

 We will be interested now in 3- and 4-dimensional skeleta of the triangulation $\Theta$ of $Y^4 \times B^{N_1}$. For notational convenience, we  denote them by  $Z ^{\epsilon}$, for $\epsilon=3$ or $4$.  These come with maps
 $$  Z^4 \overset {F= \pi |_ {Y^4}} { \xrightarrow{\hspace*{1cm}} }  Y^4, \   Z^3 \overset {f= r \circ F} { \xrightarrow{\hspace*{1cm}} } M(\Gamma). \e(23)$$

 \begin{lem}\label{lemE}
 After a small perturbation of the 0-skeleton, $\Theta ^{(0)}$, of $\Theta$, followed by a global isotopic
 perturbation of $\Theta$, which leaves it affine, we can make so that the maps
 $$ Z^4 \overset {F} {\longrightarrow}  Y^4, \ \mbox{and} \quad Z^3 \supset F^{-1} \partial Y ^4
 \overset {F |_{ F^{-1} \partial Y^4} } { \xrightarrow{\hspace*{1.4cm}} }  \partial Y^4  \e (24)$$
 are non-degenerate simplicial surjections, the restrictions of which,
 on each simplex, are affine.
 \end{lem}

 \begin{proof} The proof is left to the reader (see e.g. the argument analogous to
 this one in \cite{Po2-duke}).
 \end{proof}

 So, in the context of (24), we have now two affine triangulations, $\Theta (Z^4)$ and 
 $\Theta (Y^4)$, connected by a simplicial non-degenerate map $F$.

 We introduce now a second class of triangulations, compatible with the same
 differential structure as the $\Theta (Y^4)$, but related now to the foliation $\mathcal F$ too.
 These triangulations are denoted $\Theta _{\mathcal F}(Y^4)$, and will be subjected to
 the following conditions:

 \begin{itemize}
\item [ (24$-$1)]  $ M(\Gamma)$ is a subcomplex of  $ \Theta _{\mathcal F} (Y^4)$;

\item  [ (24$-$2)]  there is a distinguished, quite dense, set of leaves, all subcomplexes of
$\Theta _{\mathcal F} (Y^4)$,
 $$L^3_0 = M(\Gamma), \ L_1^3, L_2^3 , \cdots , L_q ^3 = L^3 \times \{1\}  = \partial Y^4,$$
  such that every 4-simplex $\sigma ^4$ of $\Theta _{\mathcal F} (Y^4)$ rests on two consecutive distinguished
  leaves $L_i ^3, L^3_{i+1}$;

\item  [ (24$-$3)] the 3-simplexes of $\Theta _{\mathcal F} (Y^4) $  are all essentially parallel to $ \mathcal F $,
 always transversal to the fibers of the retraction $r$ (from (21)), and such that $ r | {\sigma ^3} $  injects.
\end{itemize}

 \noindent Moreover, it is assumed that the triangulation $\Theta _{\mathcal F} (Y^4) | _{M(\Gamma)}$ is sufficiently fine so that
 $r (\sigma ^3)$ is a subcomplex. 

 In the context of $\Theta _{\mathcal F}$ we will have $\mathcal F$\textit{-admissible} subdivisions
 $$ \Theta _{\mathcal F} (Y^4) \overset {\mathcal F \small{-admissible}}
 {\underset {\small{subdivisions}}{ \xrightarrow{\hspace*{2cm}} }}\Theta ^1 _{\mathcal F} (Y^4) \e (25)$$

\noindent which are both admissible and respect the conditions (24$-$1) to (24$-$3), with a possibly denser,
bigger subset of distinguished leaves.

\begin{lem}\label{lemF}
Once  both $\Theta (Y^4)$ and $\Theta _{\mathcal F} (Y^4)$ are given, there exist then admissible, respectively
$\mathcal F$-admissible, subdivisions for each of them, yielding isomorphic cell-decompositions, like in
the diagram below (where all the vertical arrows are subdivisions)
$$\begin{diagram}
  & & \Theta _{\mathcal F} (Y^4)  &  \lTo^{\small{cell-decompositions}} & Y^4 & \rTo ^{\small{cell-dec.}}& \Theta  (Y^4 )& \lTo _F & \Theta  (Z^4)\\
&&\dTo   &  &  && \dTo &&  \dTo\\
M(\Gamma) & \lTo ^{\ r} & \Theta ^1_{\mathcal F} (Y^4)  &  \lTo^{\ \ \   \small{\ isomorphism}\ \mathcal I}&  & & \Theta ^1 (Y^4 )& \lTo ^{\ F^1}& \Theta ^1 (Z^4)
\end{diagram} \e (26)$$
where both $F$ and $F^1$ in the diagram are simplicial and non-degenerate.
\end{lem}

\begin{proof}

Both $\Theta (Y^4)  $ and $\Theta _{\mathcal F} (Y^4)  $ are compatible with the same \textsc{Diff} structure
on $Y^4$, and, via the smooth Hauptvermutung, they have isomorphic subdivisions. From there on,
one uses Siebenmann's cellulations and his very transparent version of the old Alexander lemma \cite{Sie_bis}.
\end{proof}

\subsection{Proof of the Main Theorem, point 1.}
\begin{proof}
By taking the universal cover  of the lower long composite arrow in (26), we get the following map
$$ X^3 \overset{\textsc{def}}= \{\mbox{the 3-skeleton of the universal cover of } \Theta ^1 (Z^4)\} \overset {f} {\longrightarrow} \widetilde {M (\Gamma)}, \e (27)$$

\noindent where $f \overset{\textsc{def}}= (r \circ \mathcal I \circ F^1)^\sim$,
which has the following features:
\begin{itemize}
\item [ (27$-$1)] since both $F^1$ and $r$ are non-degenerate, so is $f$;

\item [ (27$-$2)] we have started from $H(0)= \widetilde \Theta$ which was \textsc{wgsc} and,
from there on, all the subdivisions were admissible: this implies that $X^3$ is also \textsc{wgsc};

\item [ (27$-$3)] the map $f$  is surjective and, moreover, it admits the section $\widetilde {M(\Gamma)}
\subset X^3$ (see (24-1)), which is such that $f | _{\widetilde {M(\Gamma)} }=id$.
\end{itemize}

\noindent From this point on, there is a standard argument showing
that $\Psi (f) = \Phi (f)$ (and see here e.g. the proof of Lemma 2.8 in
\cite{Po2-duke} too). In a nutshell, this argument is the
following. Assume $\Psi (f) \subsetneq \Phi (f)$, then the induced
map
$$ X ^3 / \Psi (f) \longrightarrow \widetilde {M(\Gamma)} $$

\noindent would  have singularities, which is a contradiction.

So, by now, we have already shown that (27) is an equivariant, \textsc{wgsc}
$3^d$-\textsc{representation} of $\widetilde M(\Gamma)$. It remains to check condition (5) of Theorem \ref{thm2}, or, equivalently, (5$-bis$).

Since the fibers of $Y^4  \times B^{N_1}  \overset {r \circ \pi} {\xrightarrow{\hspace*{0.6cm}}}
M(\Gamma) $  are compact, then so are also those of $ \Theta ^4 (Z^4)
 \overset {r \circ \mathcal I \circ F^1} {\xrightarrow{\hspace*{1cm}}} M(\Gamma)$ and of $$ 
 \mbox {the 3-skeleton of } \Theta ^4 (Z^4)  \overset {r \circ \mathcal I \circ F^1} {\xrightarrow{\hspace*{1.5cm}}} M(\Gamma) . \e (28)$$

This means that, in the context of (28), for any 3-simplex $\sigma ^3$
 of $M (\Gamma)$, the inverse image consists of a finite number of
 3-simplexes, this number being clearly uniformly bounded.

 By equivariance, the same is true for
 $$X^3 \overset{f} {\longrightarrow} \widetilde {M(\Gamma)}, $$

\noindent and the point (1)  in Theorem \ref{thm2} is by now proved.
\end{proof}

\subsection{Proof of the Main Theorem, point  2.}

\begin{proof}
We want to move now from the 3-dimensional \textsc{representation}
(27) to a 2-dimensional  one
$$X^2 \overset{f} {\longrightarrow} \widetilde {M(\Gamma)}, \e (29)$$
\noindent which should be \textsc{wgsc}, equivariant, and also
satisfying (6).

The general idea is that, for the passage (27) $\Longrightarrow$ (29), there
is a similar step in \cite{Po_QSF2}, and the techniques used there can be adapted here too. Hence, we will only give here the main lines of the argument.
Since we want to have equivariance, we will work downstairs at level
$M(\Gamma)$, taking universal coverings in the end. From the lower
line in (26), we pick now the map
$$ Y^3 \overset{\textsc{def}}= \mbox { The 3-skeleton of } \Theta ^1 (Z^4) \overset {f \ \overset{\textsc{def}} =\ r \circ \mathcal I \circ F^1}
{\xrightarrow{\hspace*{1.5cm}}}  M(\Gamma), \
(\mbox{where }Y^3 = X^3 / \Gamma),  \e (30)$$

\noindent choosing to read $Y^3$ like a singular
handlebody decomposition (see here \cite{Po_QSF1_Geom-Ded, Po_QSF2, PT4-AMS}). For each
3-handle of our $Y^3$ of (30), there are three mutually orthogonal, not everywhere
well-defined foliations
$$\mathcal F_0 (\textsc{blue}), \ \mathcal F_1 (\textsc{red}),  \  \mathcal F_2 (\textsc{black}). \e (31)$$

\noindent Each 3-handle is endowed with the three foliations, but,
$\mathcal F_{\lambda} (\textsc{color})$ is \textit{natural} for
the handles of index $\lambda$. There, it is essentially a product
foliation of copies of the lateral surface of the handle in
question, namely  $\partial ($cocore$)\times $ core. The reader is
invited to look at the figures in \cite{PT4-AMS}. The paper
\cite{PT4-AMS} was written, of course, in the non-singular context
of $\widetilde M^3$ rather then of $\widetilde {M(\Gamma)}$, but,
for these individual handles, the story is the same. In  \cite{PT4-AMS}, $\widetilde M^3$ was non-singular and the three foliations were global, although of course not everywhere well-defined. While here, our $\widetilde {M(\Gamma)}$ is singular and the $3^d$ equivariant context corresponds to \cite{Po_QSF1_Geom-Ded} rather than to \cite{PT4-AMS}. Each individual handle has now its own three foliations. We can use these foliations, like in \cite{PT4-AMS}, in order to got from the $3^d$ \textsc{representation} (27) to the $2^d$ \textsc{representation} (29).

Since the map (30) is devoid of any pathologies at infinity,  we can
afford to work with usual compact handles $H^{\lambda}$, of index
$\lambda$ and dimension 3. For each of these handles $H^{\lambda}$  we consider
now a very dense 2-skeleton, which uses only finitely many leaves of the foliations (31).

Putting these things together, we get a simplicial non-degenerate map
$$ Y^2 \overset {f} {\longrightarrow } M(\Gamma) \e (32)$$

\noindent about which the following items may be assumed without any loss
of generality:

\begin{itemize}

\item [(32$-$1)]  we have subsets $Sing (f) \subset Y^2, \ Sing (Y^2) \subset Y^2 -
Sing (f)$,  just like in (3$-$6), and, outside $Sing (f)$, the double points of $f$, $M_2 (f) \subset Y^2$, 
are transversal  intersection points;

\item [(32$-$2)]  the image $ f (Y^2)  \subset M(\Gamma) $   is very dense, i.e. the
complement  $M(\Gamma) - f (Y^2)$   consists of a disjoint union of three copies
of $\mathbb R^3_+ $  glued along their common $ \partial \mathbb R^3_+ = \mathbb R^2$.

\end{itemize}

\noindent Next, we take the universal cover of (32),
$$  X^2 \overset{\textsc{def}} = \widetilde Y^2 \overset{\tilde f } {\longrightarrow } \widetilde {M(\Gamma)}. \e (33)$$

\noindent Here is what we can say about (33).

Our (33) is automatically equivariant and, since the $X^3$ in (27) was
\textsc{wgsc}, so is our present $X^2$ too, since it is actually the 2-skeleton of it.
The fact that in the context of (27) we had $\Psi (f) = \Phi (f)$, together with
the fact that $X^2$ is very dense, make that in the context of (33) we also have
$\Psi (\tilde f) = \Phi (\tilde f)$, so that (33) is an equivariant \textsc{wgsc
2$^d$-representation}, for which local finiteness should be obvious.

Locally, (33) is exactly like (32), where $Y^2$  is a finite complex. Hence, condition 
(6) of Theorem \ref{thm2} follows  automatically.  Theorem \ref{thm2} is by now completely proved.
\end{proof}

\end{document}